\documentclass[12pt]{amsart}
\usepackage{amsmath}
\usepackage{mathrsfs}
\usepackage[line,color,poly,all]{xy}
\usepackage{graphicx,hyperref,url,xy,amssymb,booktabs, verbatim}
\usepackage[noadjust]{cite}
\usepackage{xspace}
\usepackage{color}
\usepackage{psfrag}
\usepackage{mathtools}
\usepackage{verbatim}

\newcommand{\RR}{{\mathbb R}}
\newcommand{\CC}{{\mathbb C}}
\newcommand{\ZZ}{{\mathbb Z}}
\newcommand{\QQ}{{\mathbb Q}}

\newcommand{\zero}{\mathbf{0}}

\newcommand{\PPP}{{\mathscr{P}}}

\newcommand{\FFF}{{\mathscr{F}}}

\newcommand{\EEE}{{\mathscr{E}}}

\renewcommand{\tilde}{\widetilde}

\newcommand{\dee}[1]{\partial /\partial #1}
\newcommand{\fractwo}[2]{#1/ #2}

\newcommand{\card}[1]{|#1|}
\newcommand\interior{^\circ}
\newcommand{\ip}[2]{\langle #1, #2 \rangle}

\DeclareMathOperator{\disc}{disc}
\DeclareMathOperator{\cont}{cont}

\DeclareMathOperator{\Vol}{Vol}

\DeclareMathOperator{\Dim}{dim}
\DeclareMathOperator{\Codim}{codim}

\DeclareMathOperator{\Ind}{Ind}

\DeclareMathOperator{\Td}{Td}

\catcode`\@=11
\def\@secnumfont{\bfseries}
\catcode`\@12

\definecolor{blue}{rgb}{0,0,.7}

\definecolor{black}{rgb}{0,0,0}

\swapnumbers
\theoremstyle{plain}
\newtheorem{thm}[subsection]{Theorem}
\newtheorem{theorem}[subsection]{Theorem}
\newtheorem{lem}[subsection]{Lemma}

\newtheorem{prop}[subsection]{Proposition}

\theoremstyle{definition}
\newtheorem{defn}[subsection]{Definition}

\newif\ifshowvc
\showvcfalse  % uncomment this line to turn off inclusion of vc info

\ifshowvc
\fi

\begin{document}

\title[Dehn--Sommerville and Ehrhart--Macdonald]{Weighted lattice point sums in lattice polytopes,
unifying Dehn--Sommerville and Ehrhart--Macdonald}

\author{Matthias Beck}

\address{Department of Mathematics\\
San Francisco State University\\
1600 Holloway Ave\\
San Francisco, CA 94132\\U.S.A.}

\email{mattbeck@sfsu.edu}

\author{Paul E. Gunnells}

\address{Department of Mathematics and Statistics\\University of
Massachusetts\\Amherst, MA 01003-9305\\U.S.A.}

\email{gunnells@math.umass.edu}

\author{Evgeny Materov}

\address{Department of Physics, Mathematics and Computer Science\\Siberian Fire and Rescue Academy of EMERCOM of Russia\\Zheleznogorsk\\Severnaya 1\\Russia}
\email{materov@gmail.com}

\renewcommand{\setminus}{\smallsetminus}

% change before Arxiv
\date{\today} 

\thanks{
MB was partially supported by the NSF through grant DMS-1162638.
PG was partially supported by the NSF through grants DMS-1101640 and DMS-1501832.
}



\begin{abstract}
Let $V$ be a real vector space of dimension $n$ and let $M\subset V$
be a lattice.  Let $P\subset V$ be an $n$-dimensional polytope with
vertices in $M$, and let $\varphi\colon V\rightarrow \CC $ be a
homogeneous polynomial function of degree $d$.  For $q\in \ZZ_{>0}$ and any face $F$ of $P$, let
$D_{\varphi ,F} (q)$ be the sum of $\varphi$ over the lattice points
in the dilate $qF$.  We define a generating function $G_{\varphi}(q,y)
\in \QQ [q] [y]$ packaging together the various $D_{\varphi ,F} (q)$,
and show that it satisfies a functional equation that simultaneously generalizes Ehrhart--Macdonald
reciprocity and the Dehn--Sommerville relations.  When $P$ is a simple
lattice polytope (i.e., each vertex meets $n$ edges), we show how
$G_{\varphi}$ can be computed using an analogue of Brion--Vergne's
Euler--Maclaurin summation formula.
\end{abstract}

\maketitle
\ifshowvc
\let\thefootnote\relax
\footnotetext{Base revision~\GITAbrHash, \GITAuthorDate,
\GITAuthorName.}
\fi

\section{Introduction}\label{s:intro}

\subsection{} Let $V$ be a real vector space of dimension $n$ and let
$M\subset V$ be a lattice.  Let $P\subset V$ be an $n$-dimensional
\emph{polytope}, i.e., the closed convex hull of finitely many points
in $V$.  We assume further that $P$ is a \emph{lattice} polytope,
which means the vertices of $P$ lie in $M$, and that $P$ is
\emph{simple}; this means that each vertex meets $n$ edges.  (See,
e.g., \cite{ccd} for terminology and background on lattice polytopes.)
In this paper we simultaneously consider three important concepts
for~$P$:

\begin{itemize}
\item \emph{The Dehn--Sommerville relations}.  Let $\FFF$ be the set
of faces of $P$, let $\FFF (k)$ be the subset of faces of dimension
$k$, and let $f_{k} (P) = \card{\FFF (k)}$.  We define, as usual, the \emph{$h$-polynomial}
$h (P,t)=\sum_{k=0}^{n} h_{k} (P) \, t^{k}$ by
\begin{equation}\label{eq:hpoly}
h (P,t) := f_{n} (P) (t-1)^{n} + f_{n-1} (P) (t-1)^{n-1} + \dotsb + f_{0} (P) \, .
\end{equation}
(For instance, if $P$ is a simplex, then $h (P,t) =
t^{n}+t^{n-1}+\dotsb +1$.)  The Dehn--Sommerville relations say that
$h_{k} (P) = h_{n-k} (P)$ for all $k$.
\item \emph{The Ehrhart polynomial and Ehrhart--Macdonald
reciprocity}.  For any $q\in \ZZ_{> 0}$, let $qP$ denote the $q$th
dilate of $P$ and let $\EEE_{P} (q) := \card{M \cap qP}$.  Then
Ehrhart and Macdonald \cite{ehrhartpolynomial,macdonald} proved that
$\EEE_{P} (q)$ evaluates to a polynomial in $q$ that satisfies the
symmetry
\begin{equation}\label{eq:ehrrecip}
\EEE_{P} (q) = (-1)^{n}\EEE_{P\interior} (-q) \, ,
\end{equation}
where $P\interior$ is the interior of $P$.
(This holds for \emph{any} lattice polytope, not just simple ones.)
\item \emph{Euler--Maclaurin summation}.  Let $\varphi \colon V
\rightarrow \CC$ be a polynomial function.  Let $h=(h_{F})_{F\in \FFF
(n-1)}$ be a multiparameter indexed by the facets (faces of
codimension $1$) of $P$, and let $P (h)$ be the deformation of $P$
obtained by independent small parallel translations of its facets
according to $h$.  The Euler--Maclaurin formula \cite{kp,bv2} shows
how to compute the finite sum $\sum_{m\in M\cap P} \varphi (m)$ via an
explicit differential operator in the $\dee{h_{F}}$ acting on $\int_{P
(h)}\varphi (x)\,dx$, thought of as a function of $h$.
\end{itemize}
We will introduce a two-variable polynomial and prove two fundamental
theorems for it: one that simultaneously generalizes the
Dehn--Sommerville and Ehrhart--Macdonald relations, and one that gives
an Euler--Maclaurin formula.

\subsection{}
Let us be more precise about our main results.  Assume that the
polynomial $\varphi$ is homogeneous of degree $\deg{\varphi }$.
For any face $F\in \FFF$, let
\begin{equation}\label{eq:Dindividdef}
D_{\varphi ,F} (q) := \sum_{m\in M\cap qF} \varphi (m) \, .
\end{equation}
It is known that $D_{\varphi , F} (q)$ is a
polynomial in $q$ of degree $n+\deg \varphi$ and constant term $D_{\varphi , F} (0) = \varphi(0)$ \cite[Proposition~4.1]{bv1}.  Let 
\begin{equation}\label{eq:Gdef2}
G_{\varphi} (q,y) := (y+1)^{\deg{\varphi }}\sum_{F\in \FFF }
(y+1)^{\Dim F} (-y)^{\Codim F} D_{\varphi ,F} (q) \, .
\end{equation}

Our first main result is the following functional relation for the polynomial $G_{\varphi} (q,y)$.
\begin{thm}\label{thm:intro1}
\ $
G_{\varphi} (q,y) = (-y)^{n+\deg{\varphi}} \, G_{\varphi}(-q, \tfrac 1 y) \, .
$
\end{thm}
In fact, we prove a slightly more general result than Theorem
\ref{thm:intro1} that applies to all lattice polytopes $P$, simple or
not (Theorem~\ref{thm:genrecip} below).  

We now explicate how Theorem~\ref{thm:intro1} implies some of the aforementioned results.
First, suppose $\varphi = 1$ and $q=0$; then each $D_{\varphi
,F}$ equals $1$.  The generating function in Theorem~\ref{thm:intro1}
becomes
\begin{equation}\label{eq:easyG}
\sum_{F\in \FFF }
(y+1)^{\Dim F} (-y)^{\Codim F} = \sum_{k=0}^{n} (-y)^{\Dim P - k} (y+1)^{k} f_{k} (P) \, .
\end{equation}
Expanding the right of \eqref{eq:easyG} and comparing with
\eqref{eq:hpoly}, one sees that the coefficient of $y^{k}$ in
\eqref{eq:easyG} is $(-1)^{k}h_{n-k} (P)$.  Thus
Theorem~\ref{thm:intro1} in this case is equivalent to the
Dehn--Sommerville relations $h_{k} (P) =h_{n-k} (P)$.

Second, when $\varphi = 1$ and $q>0$ is a positive integer, then the
constant term of $G_{1} (q,y)$ is $\EEE_{P} (q) = \card{M\cap qP}$.
The leading term of $G_{1} (q,y)$ is an alternating sum over the face
lattice $\FFF$ of the lattice point enumerators $\EEE_{F} (q)$ and, up
to sign, nothing other than the computation of
$\EEE_{P\interior} (q)$ by inclusion-exclusion.  Thus the relation
implied by Theorem~\ref{thm:intro1} between the coefficients of
$y^{n}$ and $y^{0}$ is exactly Ehrhart reciprocity~\eqref{eq:ehrrecip}.

\subsection{}
Our second main result is a formula for $G_{\varphi} (q,y)$ in the spirit of
the \emph{Todd operator} formulas of Khovanskii--Puhklikov~\cite{kp} and
Brion--Vergne~\cite{bv2} for Euler--Maclaurin summation.  To state it we require
more notation.  Let $\langle
\phantom{a}, \phantom{b}\rangle$ be the pairing between $V$ and its
dual $V^{*}$.  Let $N \subset V^{*}$ be the lattice dual to $M$.  Any
facet $F\in \FFF (n-1)$ is the intersection of $P$ with an affine
hyperplane
\[
H_{F} = \left\{x\mid \langle x,u_{F}\rangle +\lambda_{F} =0\right\},
\]
where the normal vector $u_{F}$ is taken to be a primitive vector in
$N$.  Thus
\[
P = \left\{x\in V\mid \langle x,u_{F}\rangle +\lambda_{F} \geq 0\ \text{for
all $F\in \FFF (n-1)$}\right\}.
\]
As above, let $h=(h_{F})_{F\in \FFF (n-1)}$ be a multiparameter indexed by the
facets of $P$, and let $\tilde{P}_{q} (h)$ be the deformation by $h$ of the $q
(y+1)$ dilate of $P$:
\begin{equation}\label{eq:ineq}
\tilde{P}_{q} (h) := \left\{x\in V\mid \langle x,u_{F}\rangle +q (y+1)\lambda_{F} + h_{F} \geq 0\ \text{for
all $F\in \FFF (n-1)$}\right\}.
\end{equation}

\begin{thm}\label{thm:intro2}
There is a differential operator $\Td_{y} (P, \dee{h})$ in the
derivatives $(\dee{h_{F}})_{F\in \FFF (n-1)}$ such that
\[
G_{\varphi} (q,y) = \Td_{y} (P, \dee{h})\Bigl(\int_{\tilde{P}_{q}
(h)} \varphi (x)\, dx\Bigr)\Bigr|_{h=0}. 
\]
\end{thm}
The differential operator in Theorem \ref{thm:intro2}
will be given explicitly, after the necessary notation is
developed (see~\eqref{eq:todd.op} below and the preceding
lines).

\subsection{}
As mentioned above, we actually prove a generalization of Theorem
\ref{thm:intro1} that does not assume $P$ to be simple.  Since $q=0$
and $\varphi =1$ in Theorem~\ref{thm:intro1} recovers the Dehn--Sommerville
relations, which in turn are a manifestation of Poincar\'e duality for
the rational cohomology $H^{*} (X_{P}; \QQ)$ of the toric variety
$X_{P}$ attached to $P$ (see, e.g.,~\cite{fulton.toric}), 
it is natural to expect that the correct
generalization should somehow involve the intersection cohomology of
$X_{P}$, in other words, the \emph{$g$-polynomials}. This is indeed the case.  

It is thus natural to ask whether one can prove an analogous
generalization of Theorem \ref{thm:intro2} for general lattice
polytopes $P$.  Work of Brion--Vergne \cite{bv2} gives an analogue of
the Euler--Maclaurin formula for such polytopes, and when applied to
our setup gives explicit Todd operator formulas for the leading and
constant terms (in $y$) of $G_{\varphi } (q,y)$.  Their technique is
to consider simple deformations $P'$ of $P$ and then to take the limit
as one collapses $P'$ back down to $P$.  However, this does not lead
to a Todd operator formula for the other terms of $G_{\varphi } (q,y)$
in general.  It would be interesting to generalize the results of
\cite{bv2} to the generating function $G_{\varphi } (q,y)$.

\subsection{Acknowledgments}\label{ss:thanks}
We thank David Cox and Mich{\`e}le Vergne for helpful comments.  Two
of us (MB and EM) thank the Max Planck Institute for Mathematics for
its hospitality, where some of these results were initially worked
out.  We thank Toru Ohmoto, who informed us of \cite{ohmoto.masters}
after this work was completed.  Finally, we thank the anonymous
referees for their comments.

\section{The Reciprocity Theorem}\label{s:rl}

\subsection{}
The goal of this section is to modify \eqref{eq:Gdef2}
for a general lattice polytope such that Theorem \ref{thm:intro1}
holds.  We begin by recalling some notation.  For more details, see, e.g.,
\cite[\S3.14]{stanley}.

Let $P$ be a general polytope of dimension $n$ (not necessarily
a lattice polytope).  As above, let $\FFF$ be its set of faces, regarded as a
poset under inclusion.  We enlarge $\FFF$ to $\FFF^{-}$ by adjoining
an extra element $\zero$ that is defined to be smaller than any $F\in
\FFF$; the element $\zero$ should be thought of as corresponding to
the ``empty'' face of $P$ with dimension $\Dim \zero = -1$.  We make
$\FFF^{-} $ into a ranked poset with rank function $\rho$ by by
putting $\rho (F) = \Dim F + 1$ and $\rho (\zero) = 0$.

\subsection{}
\label{subs:g_pol}
We define polynomials $f_{P}, g_{P} \in \ZZ [x]$ as
follows:\footnote{This (standard) definition of the $f$ polynomial is
dual to the definition of the $h$-polynomial \eqref{eq:hpoly}.  The
$f$-polynomial favors simplicial polytopes, in that Dehn--Sommerville
holds with no $g$-polynomial corrections.  The $h$-polynomial, on the
other hand, favors simple polytopes.
}
\begin{itemize}
\item If $\rho (P) = 0$, we put $f_{P} (x) = g_{P} (x) = 1$.
\item Otherwise, if $\rho (P) =n+1 > 0$, then $f_{P} (x)$ is a
polynomial $\sum_{l=0}^{n} f_{l} \, x^{l}$ of degree $n$.  We recursively define
\[
g_{P} (x) = f_{0} + (f_{1}-f_{0})x + (f_{2}-f_{1})x^{2} + \dotsb +
(f_{m}-f_{m-1})x^{m},
\]
where $m=\lfloor n/2 \rfloor$, and
\begin{equation}\label{eq:fdef}
f_{P} (x) = \sum_{\zero \leq F \lneq P } g_{F} (x) (x-1)^{n-\rho (F)}.
\end{equation}
Note that the sum is taken over \emph{proper} faces of $P$, which
makes $f_{P}$ well defined by induction.
\end{itemize}

With this setup, the following \emph{master duality} theorem for the
polynomial $f_{P}$ holds (see, e.g.,~\cite[Theorem 3.14.9]{stanley}):

\begin{theorem}\label{thm:masterduality}
Let $n=\Dim P$. Then
\[
f_{P} (x) = x^{n} f_{P}(\tfrac 1 x) \, .
\]
Equivalently, if $f_{P} = \sum_{i=0}^{n} a_{i} \, x^{i}$, then $a_{i} = a_{n - i}$.
\end{theorem}

We will also need the following identity of the $f$ and $g$
polynomials:
\begin{lem}\label{lem:identity}
Let $P$ be a polytope of dimension $n$.  
Then 
\begin{equation}\label{eq:lemid}
x^{\Dim P + 1} g_{P} (\tfrac 1 x) = \sum_{\zero \leq F \leq P} g_{F} (x) (x-1)^{n-\Dim F}
\end{equation}
where the sum is taken over all faces of $P$, including $P$ itself.
\end{lem}

% \begin{ex}\label{ex:square}
% \emph{Sanity check.}  Let $P$ be the square.  Then $\Dim P = 2$ and $g_{P} = 1+x$.  So
% the identity is 
% \[
% x^{3} (1+\frac 1 x) = (1+x) + 4(x-1)+4 (x-1)^{2} + (x-1)^{3}.
% \]
% \end{ex}

\begin{proof}
The proof is a simple computation and arises in the proof of
\cite[Theorem~3.14.9]{stanley}.  Indeed, with $f_{P} = \sum_{i=0}^{n} a_{i}x^{i}$,
\[
g_{P} + (x-1)f_{P} = (a_{m}-a_{m+1})x^{m+1}+ (a_{m+1}-a_{m+2})x^{m+2}+\dotsb 
\]
where $m=\lfloor n/2 \rfloor$.  Applying Theorem~\ref{thm:masterduality},
\[
g_{P} + (x-1)f_{P} = x^{n+1}g_{P} (\tfrac 1 x) \, .
\]
Inserting the definition \eqref{eq:fdef} of $f_{P}$ completes the proof.
\end{proof}

\subsection{} Now assume that $P$ is a lattice polytope.  For any face
$F\leq P$, let $\PPP_{P}(F)$ be the dual face of $F$ in the polar
polytope to $P$.  For example, if $P$ is simple, $\PPP_{P} (F)$ is a
simplex for any proper face $F$.  We define a polynomial $\tilde{g}_{F}(x)$ by
\[
\tilde{g}_{F} (x) = g_{\PPP_{P} (F)} (x) \, .
\]
Note that $\tilde{g}_{F}$ depends on the larger polytope $P$ in which
$F$ is a face, although this is not part of the notation.  As in the
introduction, let $\varphi$ be a homogeneous polynomial and define
$D_{\varphi ,F} (q)$ by \eqref{eq:Dindividdef}.  We extend the
definition~\eqref{eq:Gdef2} of $G_{\varphi} (q,y)$ by
\begin{equation}\label{eq:Gdef3}
G_{\varphi} (q,y) := (y+1)^{\deg{\varphi }}\sum_{F\in \FFF }
(y+1)^{\Dim F} (-y)^{\Codim F} D_{\varphi ,F} (q) \, \tilde{g}_{F} (-\tfrac 1 y) \, .
\end{equation}
Note that if $P$ is simple then $\tilde{g}_{F} =1$ for all faces of $P$, and this
definition coincides with~\eqref{eq:Gdef2}.\footnote{We remark that the factor $(y+1)^{\deg{\varphi }}$
is not really needed for $G_{\varphi}$, at least as far as the
results in this section are concerned.  This factor appears naturally
when one considers the Todd operator formula, so it is reasonable to
include it here.}

\begin{thm}\label{thm:genrecip}
For a general lattice polytope $P$, the function $G_{\varphi} (q,y)$
satisfies the relation in Theorem~\ref{thm:intro1}.
\end{thm}

We shall need the following lemma:
\begin{lem}\label{lem:genehr}
Let $P$ be a lattice polytope of dimension $n$ and let $\varphi$ be a
homogeneous polynomial function.  Let $q > 0$ be an integer.  Define
\begin{align*}
D_{\varphi , P} (q) &:= \sum_{m\in M\cap qP} \varphi (m) \, , \\
D\interior_{\varphi , P} (q) &:= \sum_{m\in M\cap (qP)\interior} \varphi (m) \, ,
\end{align*}
where $P\interior$ denotes the interior of $P$.  
Then as functions of $q$, both $D$ and $D\interior$ are polynomials of
degree $\deg \varphi + \Dim P$, and  
\begin{equation}\label{eq:genrecip}
D_{\varphi , P} (-q) = (-1)^{\deg \varphi + \Dim P} D\interior_{\varphi , P} (q) \, .
\end{equation}
\end{lem}

\begin{proof}
These statements are proved by Brion--Vergne in 
\cite[Proposition 4.1]{bv1} for any simple lattice polytope.  Their
later paper \cite{bv2} derives an Euler--Maclaurin formula for any
general lattice polytope $P$ by first passing to a simple perturbation
$P'$ and computing on $P'$ as in \cite{bv1}.  This implies the result.
\end{proof}

\begin{proof}[Proof of Theorem \ref{thm:genrecip}]
Let $n$ be the dimension of $P$ and $d$ the degree of $\varphi$.
Write $G = G_{\varphi} (q,y)$ and $G' = (-y)^{n+d} G (-q,\tfrac 1 y)$.
We begin with the definition
\[
G = (y+1)^{d}\sum_{F\in \FFF }
(y+1)^{\Dim F} (-y)^{\Codim F} D_{\varphi ,F} (q) \tilde{g}_{F} (-\tfrac 1 y)
\]
and replace each $D_{\varphi , F}$ with the sum over the faces of $F$
of the functions $D\interior_{\varphi }$ to obtain
\[
G = (y+1)^{d}\sum_{F\in \FFF }
(y+1)^{\Dim F} (-y)^{\Codim F} \, \tilde{g}_{F} (-\tfrac 1 y) \sum_{E\leq F} D\interior_{\varphi ,E} (q) \, .
\]
After interchanging the sums and swapping the labels of $E$ and $F$, 
\begin{equation}\label{eq:gee}
G = (y+1)^{d} \sum_{F\leq P} D\interior_{\varphi , F} (q) \sum_{F\leq E
\leq P} (y+1)^{\Dim E} (-y)^{\Codim E} \, \tilde{g}_{E} (-\tfrac 1 y) \, .
\end{equation}
Now consider $G'$.  If we apply Lemma \ref{lem:genehr} then 
\begin{equation}\label{eq:geeprime}
G' = (y+1)^{d}\sum_{F\leq P} (y+1)^{\Dim F} D\interior_{\varphi , F} (q) \, \tilde{g}_{F} (-y) \, .
\end{equation}
Comparing \eqref{eq:gee} and \eqref{eq:geeprime}, we see that we need
the following identity for any face $F$ of~$P$:
\begin{equation}\label{eq:id}
(y+1)^{\Dim F} \, \tilde{g}_{F} (-y) = \sum_{F\leq E\leq P} (y+1)^{\Dim
E} (-y)^{\Codim E} \, \tilde{g}_{E} (-\tfrac 1 y) \, .
\end{equation}
We claim that this follows from Lemma \ref{lem:identity}.  To see this,
one observes that the polynomial $\tilde{g}_{F}$ is the $g$-polynomial
of the dual face $\PPP_{P} (F)$, and that the sum over $F\leq
E \leq P$ is the same as the sum over the face poset for
$\PPP_{P}(F)$.  Applying this and putting $x=-y$ gives~\eqref{eq:lemid}.
\end{proof}

\subsection{} Notice that the reciprocity law (\ref{eq:geeprime}) suggests another definition of the 
polynomial $G_\varphi(q, y)$ from (\ref{eq:Gdef2}) 
\[
G_{\varphi} (q,y) := (y+1)^{\deg\varphi}
\sum_{F\in \mathscr{F} }
(y+1)^{\dim F} D_{\varphi ,F}^\circ (q) \, 
\]
and its extended version from (\ref{eq:Gdef3})
\begin{equation}
\label{eq:reciprG}
G_{\varphi} (q,y) := (y+1)^{\deg\varphi}
\sum_{F\in \mathscr{F} }
(y+1)^{\dim F} D_{\varphi ,F}^\circ (q) \, \tilde{g}_{F} (-y) \, .
\end{equation}

\section{The Todd Operator Formula}\label{s:tof}

\subsection{}
For the rest of the paper we assume that $P$ is simple.  We begin by
introducing the notation we need to define the Todd operator.

Let $f\in \FFF (n-l)$ be a face of codimension $l$, and let $H_{f}$ be
the affine subspace spanned by $f$.  Since $P$ is simple, there are
exactly $l$ hyperplanes in $\{H_{F}\mid F\in \FFF (n-1) \}$ whose
intersection is $H_{f}$.  Let $\sigma_{f}\subset V^{*}$ be the convex cone
generated by the corresponding normal vectors $\{u_{F} \mid F \in \FFF
(n-1), F\supset f\}$.  
The cone
$\sigma_{f}$ is called the \emph{normal cone} to~$f$.

\subsection{}
The set $\Sigma = \{\sigma_{f}\mid f\in \FFF \}$ of all normal cones forms an acute rational
po\-ly\-hedral fan in $V^{*}$.  This means the following:
\begin{enumerate}
\item Each $\sigma \in \Sigma$ contains no nontrivial linear subspace.
\item If $\sigma'$ is a face of $\sigma \in \Sigma$, then
$\sigma '\in \Sigma$.  
\item If $\sigma$, $\sigma '\in \Sigma$, then $\sigma \cap \sigma '$
is a face of each.
\item Given $\sigma \in \Sigma$, there exists a finite set $S\subset
N$ such that any point in $\sigma$ can be written as $\sum
\rho_{s}s$, where $s\in S$ and $\rho_{s}\geq 0$.  
\end{enumerate}
Moreover, $P$ being simple implies that $\Sigma$ is simplicial, which means that
in (4) we can take $\# S = \Dim \sigma$ for all $\sigma$.  The fan
$\Sigma$ is called the \emph{normal fan} to~$P$.

\subsection{} Let $\rho \in \Sigma$ be a rational $1$-dimensional
cone.  Then $\rho$ contains a unique primitive point, which we call
the \emph{spanning point} of $\rho$.  For any cone $\sigma $,
we denote by $\sigma (1)$ the set of spanning points of all
$1$-dimensional faces of $\sigma$ and write
\[
\Sigma (1) :=
\bigcup_{\sigma \in \Sigma} \sigma (1) \, .
\]
There is bijection between $\Sigma (1)$ and $\FFF (n-1)$: if $\rho\in
\Sigma (1)$, then the spanning point of $\rho$ is a unique normal
vector $u_{F}$, which determines the corresponding facet $F$.

For any cone $\sigma \in \Sigma $, let $U (\sigma)$ be the sublattice of
$N$ generated by the spanning points of $\sigma$.  Set
\[
  N (\sigma) := N\cap (U (\sigma)\otimes \QQ )
  \qquad \text{ and } \qquad
  \Ind \sigma := [N(\sigma )\colon U (\sigma)] \, .  
\]
If $\Ind \sigma = 1$, then $\sigma$ is called \emph{unimodular}.  The
polytope $P$ is called \emph{nonsingular} if % and only if all its
normal cones are unimodular.\footnote{This condition is the same as
the toric variety $X_{P}$ determined by $P$ being nonsingular
\cite{fulton.toric}.}  Let $G (\sigma)$ be the finite group $N (\sigma)/U
(\sigma)$.

\subsection{}
For any $\sigma\in \Sigma$, define 
\[
Q (\sigma) := \left\{\sum_{s\in \sigma (1)} \rho_{s}s \bigm |0\leq
\rho_{s} < 1\right\}.
\]
Note that $\Vol Q (\sigma) = \Ind \sigma$, and $Q (\sigma)\cap 
N(\sigma)=\{0 \}$ if and only if $\sigma$ is unimodular.  Furthermore,
the set $Q (\sigma) \cap N (\sigma)$ is in bijection with the finite group $G
(\sigma)$ under the map $N (\sigma)\rightarrow N (\sigma)/U
(\sigma)$.  Put
\[
\Gamma_{\Sigma } := \bigcup_{f\in \FFF} Q (\sigma_{f})\cap N.
\]
Then $\Gamma_{\Sigma}=\{0 \}$ if and only if $P$ is nonsingular.

\subsection{} As in the introduction, let $y$ be a real variable,
and let $h= (h_{F})_{F\in \FFF (n-1)}$ be a real multivariable indexed
by the facets of $P$.  As before let $\tilde{P}_q (h)$ be the deformation by $h$
of the $q (y+1)$ dilate of $P$ defined in~\eqref{eq:ineq}.
The polytope $\tilde{P}_q (h)$ depends on $y$, but we suppress
this from the notation.  

If $q=1$ and $y=0$ then $\tilde{P}_1 (0) = P$;
furthermore if $q\not =0$ and $y\not = -1$, then $\tilde{P}_q (h)$ is
isomorphic to $P$ for small $h$; in this case the integral
\begin{equation}\label{eq:bvEdef}
I (\tilde{P}_q (h)) = I_{\varphi } (\tilde{P}_q (h)) := \int_{{\tilde{P}_q(h)}} \varphi (x)\,dx
\end{equation}
therefore converges for small $h$ (here we take the measure on $V$ that gives a
fundamental domain of $M$ unit volume).  We will compute the function
$G_{\varphi} (q,y)$ 
by applying a differential
operator to $I (\tilde{P}_q (h))$, the Todd-$y$ operator.  To define it,
we need yet more notation.

\subsection{}
For each facet $F\in \FFF (n-1)$, let $\xi_F\colon V^{*} \rightarrow \RR$ be
the unique piecewise-linear continuous function defined by
\begin{itemize}
\item $\xi_{F} (s) = 1$ if $s\in \Sigma (1)$ is the spanning point
corresponding to $F$, 
\item $\xi_{F} (s') =0$ for all other $s'\in \Sigma (1)$, and 
\item $\xi_{F}$ is linear on all the cones of $\Sigma $.
\end{itemize}
Put $a_F (x) = \exp (2\pi i \, \xi_{F} (x))$ for all $x\in
V $.   

Suppose $g\in \Gamma_{\Sigma}\cap \sigma$.  Then the pair
$(g,\sigma)$ determines a tuple of roots of unity as follows.  If
$s_{1},\dotsc ,s_{l}$ are the spanning points of $\sigma$, and
$F_{1},\dotsc ,F_{l}$ are the corresponding facets, then we can attach
to $(g,\sigma)$ the tuple $(a_{1} (g),\dotsc ,a_{l} (g))$, where we
have written $a_{i}$ for $a_{F_{i}}$.  We are now ready to define the
Todd-$y$ operator:

\begin{defn}\label{def:toddy}
Let $a$ be a complex number and $x$ a real variable.  We define
$\Td_{y} (a,\dee{x})$ to be the differential operator given formally
by the power series
\begin{multline*}
\frac{\dee{x}(1+ay\exp(-\dee{x} (y+1)))}{1-a\exp(-\dee{x}(y+1))} =\\
\frac{(y+1)\dee{x}}{1-a\exp(-\dee{x}(y+1))}-y \, \dee{x} = \sum_{k=0}^{\infty} c (a,k,y)
\left(\frac{\partial }{\partial x}\right)^{k}.
\end{multline*}
\begin{table}[htb]
\begin{center}
\begin{tabular}{|c||l|}
\hline
$k$&$c (a,k,y)$\\
\hline\hline
$1$&$-\fractwo{a (y+1)}{a-1}$\\
$2$&$-\fractwo{a (y+1)^2}{(a-1)^2}$\\
$3$&$-\fractwo{a (a+1) (y+1)^3}{2 (a-1)^3}$\\
$4$&$-\fractwo{a (a^2+4 a+1) (y+1)^4}{6 (a-1)^4}$\\
$5$&$-\fractwo{a (a^3+11 a^2+11 a+1) (y+1)^5}{24 (a-1)^5}$\\
\hline
\end{tabular}
\end{center}
\bigskip
\caption{\label{tab:cir}Sample coefficients $c (a,k,y)$.}
\end{table}
Table \ref{tab:cir} gives some examples of the polynomials $c (a,k,y)$.
We remark that  $c (1,k,0)=B_{k}/k!$, where $B_{k}$ is the $k$-th Bernoulli
number.\footnote{With this convention the Bernoulli numbers are $B_{0} = 1$, 
$B_{1}=\frac 1 2$, $B_{2}=\frac 1 6$, $B_{4}=-\frac 1 {30}$, \dots , and $B_{2k-1}=0$ for
$k > 1$.  Note that for many authors $B_{1}=-\frac 1 2$.}
If $a\not =1$, then 
\[
-(k-1)!(a-1)^{k}c (a,k,y)/a(y+1)^{k}
\]
is the \emph{Eulerian polynomial} for the symmetric group $S_{k-1}$ (see, e.g.,~\cite{hirzeuler}).

\subsection{}
Recall that $h$ is a multivariable with components $h_{F}$
indexed by the facets of $P$.  For any $g\in \Gamma_{\Sigma}$, we define
\[
\Td_{y} (g,\dee{h}) := \prod_{F\in \FFF (n-1)} \Td_{y} (a_{F} (g),\dee{h_{F}})
\]
and 
\begin{equation}\label{eq:todd.op}
\Td_{y}(P ,\dee{h}) := \sum_{g\in \Gamma_{\Sigma}} \Td_{y} (g,\dee{h}).
\end{equation}
\end{defn}

This concludes our setup and makes the statement of Theorem~\ref{thm:intro2} precise.
We now turn to its proof.

% We can state the main theorem of this section:
% 
% \begin{thm}\label{thm:main}
% Suppose $P$ is a simple lattice polytope.  Then 
% \[
% G_{\varphi} (q,y) = \Td_{y} (P, \dee{h}) E_{\varphi} (\tilde{P} (h))\Bigr|_{h=0}. 
% \]
% % Furthermore, $G_{\varphi} (q,y)$ satisfies the reciprocity theorem 
% % \[
% % G_{\varphi} (q,y) = (-y)^{n+\deg{\varphi}} G_{\varphi}(-q,\tfrac 1 y).
% % \]
% \end{thm}
% 
% % \subsection{} We prove Theorem \ref{thm:main} in the next
% % section. \fixme{add something about original B-V formula.}

\section{Proof of Theorem \ref{thm:intro2}}\label{s:proof}

To prove Theorem \ref{thm:intro2} we adapt arguments in \cite{bv1} to
incorporate the parameter $y$.  
% First we introduce more notation.  
For any face $f \in \FFF$, let $C_{f} \subset V$ be the convex cone
generated by elements $p-p'$ with $p\in P$ and $p'\in f$.  The cone
$C_{f}$ is called the \emph{tangent cone} to $P$ at $f$.  The normal
cone $\sigma_{f}$ is the dual cone to $C_{f}$.  We also denote by
$\FFF^{f} \subset \FFF (n-1)$ the subset of facets of $P$ containing
$f$.

Let $V_{\CC} = V\otimes \CC$ be the complexification of $V$, and let
$V^{*}_{\CC}$ be its dual space.    We extend the pairing $\langle
\phantom{a}, \phantom{a} \rangle$ to $V_{\CC}$ and $V^{*}_{\CC}$.  Let
$z\in V_{\CC}^{*}$ and consider the integral 
\[
I (P) (z) := \int_{P} \exp {\ip{x}{z}}\,dx
\]
% and the exponential sums
% \[
% D (P) (y) = \sum_{m\in M\cap P} \exp {\ip{m}{y}}, \quad D (P\interior
% ) (y) = \sum_{m\in M\cap P\interior } \exp {\ip{m}{y}}.
% \]
and the exponential sum
\[
D (P) (z) := \sum_{m\in M\cap P} \exp {\ip{m}{z}} \, .
\]
Brion--Vergne \cite{bv1} gave explicit formulas for $I (P)$ and $D (P)$ for
generic $z$; we recall them here.  For any vertex $v\in P$, we have
the normal cone $\sigma_{v}$ with spanning points $\{u_{F}\mid F\in
\FFF^{v} \}$.  Let $\{m_{v}^{F}\mid F\in \FFF^{v} \}$ be the dual
basis.  The points $m_{v}^{F}$ are rational generators for the tangent
cone $C_{v}$ and, in particular, lie along the edges of $P$ through
$v$.  Let $M (v)\subset V$ be the lattice they generate.  Then any
$\gamma \in G (\sigma_{v}) = N/U (\sigma_{v})$ determines a character
$\chi_{\gamma } \colon M (v)/M\rightarrow \CC^{\times}$ via
\[
  \chi_{\gamma} (m) = \exp (2\pi i \ip{m}{\tilde{\gamma}}) \, ,
\]
where $\tilde{\gamma} \in N$ is any representative of~$\gamma$.

\begin{prop}\label{prop:bvexp}
For $z\in V^{*}_{\CC}$ generic,
\begin{equation}\label{eq:bvidD}
D (P) (z) = \sum_{v\in \FFF (0)} \frac{\exp{\ip{v}{z}}}{|G
(\sigma_{v})|} \sum_{\gamma \in G (\sigma_{v})} \prod_{F\in \FFF^{v}}
\frac{1}{1-\chi_{\gamma } (m_{v}^{F}) \exp {\ip{m_{v}^{F}}{z}}}
\end{equation}
and
\begin{equation}\label{eq:bvidE}
I(P)(z) = (-1)^{n}\sum_{v\in \FFF (0)} \exp \ip{v}{z} \left(|\det
(m_{v}^{F})|_{F\in \FFF^{v}}\right)\prod_{F\in
\FFF^{v}}\frac{1}{\ip{m_{v}^{F}}{z}} \, .
\end{equation}
\end{prop}

\begin{proof}
See \cite[Propositions 3.9 and 3.10]{bv1}.
\end{proof}

\begin{lem}\label{lem:faceid}
Let $E\in \FFF$ be a face of $P$, let $\FFF_{E} \subset \FFF$ be
the subset of faces of $E$, and let $D (E) (z) = \sum_{m\in M\cap
E} \exp \ip{m}{z}$.  Then for $z$ generic, 
\[
D (E) (z) = \sum_{v\in \FFF_{E} (0)} \frac{\exp{\ip{v}{z}}}{|G
(\sigma_{v})|} \sum_{\gamma \in G (\sigma_{v})} \prod_{\substack{F\in
\FFF^{v}\\F\supsetneq E}} \frac{1}{1-\chi_{\gamma } (m_{v}^{F})
\exp {\ip{m_{v}^{F}}{z}}} \, .
\]
\end{lem}

\begin{proof}
The proof is essentially the same as that of \eqref{eq:bvidD}.  The
main point is that if one considers a single vertex $v$ in
\eqref{eq:bvidD}, then the sum over $G (\sigma_{v})$ induces character
sums that equal $1$ on $C_{v}\cap M$ and 0 on $C_{v} \cap (M (v)
\smallsetminus M)$.  These sums have the same effect on $M \cap E$ for
any face $E\subset P$.  Furthermore, for any vertex $v$ of $E$, the
points $m_{v}^{F}$ in the dual basis lie along edges of $E$ exactly
for the facets $F$ \emph{not} containing $E$.
\end{proof}

Now we build an exponential version of our generating function:
\begin{equation}\label{eq:tildeG}
\tilde{G}_{\disc} (z, y) := \sum_{E\in \FFF} (y+1)^{\Dim E} (-y)^{\Codim E} D (E) (z).
\end{equation}

\begin{lem}\label{lem:prodfmla}
For $z$ generic, 
\begin{equation}\label{eq:Gdiscverts}
\tilde{G}_{\disc} (z, y) =  \sum_{v\in \FFF (0)} \frac{\exp{\ip{v}{z}}}{|G
(\sigma_{v})|} \sum_{\gamma \in G (\sigma_{v})} \prod_{F\in \FFF^{v}}
\Bigl(\frac{y+1}{1-\chi_{\gamma } (m_{v}^{F}) \exp {\ip{m_{v}^{F}}{z}}} - y\Bigr).
\end{equation}
\end{lem}

\begin{proof}
This follows from Lemma \ref{lem:faceid} and the fact that $P$ is
simple.  Indeed, consider expanding the products over the sets
$\FFF^{v}$.  At each vertex $v$ one sees products over all possible
subsets of the edges emanating from $v$.  Each subset determines a
unique face containing $v$.  If we take a face $E$ and collect
the terms corresponding to these edge subsets for the
vertices of $E$, we obtain exactly the expression in Lemma~\ref{lem:faceid} for~$D (E) (z)$.
\end{proof}

Next we consider an integral version of $\tilde{G}_{\disc} (z, y)$.
As before, let $h=(h_{F})_{F\in \FFF (n-1)}$ be a multiparameter
indexed by the facets of $P$, and recall (cf. \eqref{eq:ineq})
that $\tilde{P}_{1} (h)$ is the deformation by $h$ of the
$(y+1)$-dilate of $P$:
\[
\tilde{P}_{1}(h) = \left\{x\in V\mid \langle x,u_{F}\rangle +(y+1)\lambda_{F} + h_{F} \geq 0\ \text{for
all $F\in \FFF (n-1)$}\right\}.
\]
Given any vertex $v\in P$, the corresponding vertex in $\tilde{P}_{1} (h)$
is 
\[
v (h) = (y+1) v-\sum_{F\in \FFF^{v}} h_{F} \, m_{v}^{F}.
\]
We define 
\[
\tilde{G}_{\cont} (z, y) := I (\tilde{P}_1 (h)) (z) = \int_{\tilde{P}_1 (h)} \exp \ip{x}{z} \, dx.
\]

\begin{lem}\label{lem:Gcont1}
We have 
\begin{equation}\label{eq:Gcontverts}
\tilde{G}_{\cont} (z, y) = (-1)^{n}\sum_{v\in \FFF (0)} \frac{\exp
\ip{(y+1) v - \sum_{F\in \FFF^{v}} h_{F} m_{v}^{F}}{z}}{\card{G (\sigma_{v})}} \prod_{F\in
\FFF^{v}}\frac{1}{\ip{m_{v}^{F}}{z}} \, .
\end{equation}
\end{lem}

\begin{proof}
This follows from \eqref{eq:bvidE} with $P$ replaced by $\tilde{P}_1 (h)$,
together with the observation that  $1/\card{G (\sigma_{v})} = |\det
(m_{v}^{F})|_{F\in \FFF^{v}}$.
\end{proof}

We now consider the action of the operator $\Td_{y} (P,\dee{h})$ on
$\tilde{G}_{\cont}$.  In particular we will compute the action on the
terms for the different vertices in \eqref{eq:Gcontverts}
and will ultimately compare the result with the corresponding terms in 
\eqref{eq:Gdiscverts}.  Put 
\[
\tilde{G}_{\cont} (v, z, y) := \frac{\exp
\ip{(y+1) v - \sum_{F\in \FFF^{v}} h_{F} m_{v}^{F}}{z}}{\card{G (\sigma_{v})}} \prod_{F\in
\FFF^{v}}\frac{1}{\ip{m_{v}^{F}}{z}} \, .
\]

\begin{lem}\label{lem:certaingamma}
Let $\gamma \in \Gamma_{\Sigma}$ and let $y$ be generic. Then $\Td_{y} (\gamma , \dee{h})
\tilde{G}_{\cont} (v, z, y) =0$ unless $\gamma \in \sigma_{v}$.  In
the latter case,
\begin{multline}\label{eq:singterm}
\Td_{y} (\gamma , \dee{h})\tilde{G}_{\cont} (v, z, y)\bigr|_{h=0} =\\
\frac{\exp \ip{(y+1)v}{z}}{\card{G (\sigma_{v})}} \prod_{F\in
\FFF^{v}} \Bigl(\frac{y+1}{1-a_{F} (\gamma)\exp ( (y+1)\ip{m_{v}^{F}}{z})}-y\Bigr).
\end{multline}
\end{lem}
\begin{proof}
The first statement is proved in \cite[Proof of Theorem 3.12]{bv1}. The
second follows from a direct computation using the identity (with
$a\in \CC$, $x$ and $u$ real variables)
\[
\Td_{y} (a,\dee{x})\exp xu \bigr|_{x=0} = \frac{u (y+1)}{1-a\exp (-u (y+1))}-uy. \qedhere
\]
\end{proof}

\begin{theorem}\label{thm:expsumid}
Let $z$ be generic. Then 
\begin{equation}\label{eq:cisd}
\Td_{y} (P,\dee{h})\tilde{G}_{\cont} (z,y)\bigr|_{h=0} = \tilde{G}_{\disc} ((y+1)z, y).
\end{equation}
\end{theorem}

\begin{proof}
This follows from comparison of Lemmas \ref{lem:prodfmla} and
\ref{lem:certaingamma}.  Indeed, by Lemma \ref{lem:certaingamma} only
the $\gamma$ giving elements in $G (\sigma_{v})$ are relevant for
computing $\Td_{y} (P, \dee{h})$ on $\tilde{G}_{\cont} (v, z, y)$.
Furthermore, if $\gamma \in G (\sigma_{v})$ and $F\in \FFF (n-1)$ contains $v$, then a direct computation
shows 
\[
a_{F} (\gamma) = \chi_{\gamma} (m_{v}^{F}) \, .
\]
Thus we have equality in the vertex contributions to each side of
\eqref{eq:cisd}, after we replace $z$ in $\tilde{G}_{\disc}$ with $(y+1)z$.
\end{proof}

\begin{proof}[Proof of Theorem \ref{thm:intro2}]
We take the Taylor expansion on both sides of
\eqref{eq:cisd} with respect to $z$, after replacing the deformed dilate
$\tilde{P}_1 (h)$ with the $h$-deformation of the $(y+1)$-dilate of $qP$,
which is $\tilde{P}_q (h)$.
\end{proof}

\section{Relation to the Hirzebruch--Riemann--Roch Theorem}

In recent years, a bridge between % geometry of toric varieties and
geometry has allowed one to prove beautiful results in geometry and
combinatorics using tools from algebraic geometry. Many combinatorial
results have their avatars in algebraic geometry and vice versa.  In
particular, the polynomial $G_{\varphi} (q, y)$ can be regarded as a 
generalization of the Hirzebruch $\chi_y$-genus for a singular toric
variety.

In this section, we show that Theorem~\ref{thm:intro2} agrees with the
representation of the normalized Hirzebruch class of a toric variety
studied by Maxim--Sch\"uhrmann \cite{paper:M-S}.  Since we treat
lattice polytopes whose toric varieties are not necessarily smooth, we
have to involve different approaches to the study of singular
varieties such as orbifolds, motivic approach, intersection homology
theory, etc.

\subsection{} Let $X = X_\Sigma$ be a complete toric variety of
dimension $n$ defined by the fan $\Sigma$. Denote by
$\widehat{\Omega}_X^p$ the sheaf of Zariski differential $p$-forms on
$X$. Recall that $\widehat{\Omega}_X^p$ is defined as
$\widehat{\Omega}_X^p := i_*\Omega_U^p$, where $i\colon
U\xhookrightarrow{} X$ is the inclusion of the nonsingular locus $U$
into $X$. Given an ample Cartier divisor $D$ on $X$, let
$\mathscr{O}_X(D)$ be the corresponding invertible. Let $P = P_D$ be
the \textit{support polytope of $D$}. For now, we suppose that the
class of $D$ is nontrivial in the Picard group of $X$.

The \textit{$\chi_y$-characteristic} (or \textit{generalized Hirzebruch polynomial of $D$}) is defined by 
\begin{align*}
\chi_y(X, \mathscr{O}_X(D)) &:= \sum_{p\ge 0} \chi(X, \widehat{\Omega}_X^p\otimes \mathscr{O}_X(D))\cdot y^p \\
&= \sum_{p\ge 0} \left( \sum_{i\ge 0} (-1)^i \dim_{\mathbb{C}} H^i (X,
\widehat{\Omega}_X^p\otimes \mathscr{O}_X(D)) \right) y^p .
\end{align*}
In particular, the \textit{$\chi_y$-genus} of a toric variety is defined as
\[
\chi_y(X) := \sum_{j, p \ge 0} (-1)^{j-p}
\dim_{\mathbb{C}}\mathrm{Gr}_F^p H_c^j(X; \mathbb{C}) \cdot y^p,
%\sum_{p\ge 0} \chi(X, \widehat{\Omega}_X^p)\cdot y^p.
\]
where $F$ denotes the Hodge-Deligne filtration on $H_c^j(X; \mathbb{C})$.

The combinatorial expression for $\chi_y(X, \mathscr{O}_X(D))$ in
terms of weighted sums of numbers of lattice points in faces of the
polytope $P_D$ was first obtained in \cite{paper:Materov_Bott} and
also reproved in \cite[Corollary~4.3]{paper:M-S}.

\begin{theorem}
\label{thorem:Bott} Let $X$ be a complete simplicial toric variety
with ample Cartier divisor $D$. Then the $\chi_y$-characteristic has
the following combinatorial representation in terms of sums of lattice
points over faces of the support polytope $P$ of $D$:
\begin{align}
\label{eq:Bott_formula}
\chi_y(X, \mathscr{O}_X(D)) &=
\sum_{F\in\mathscr{F}_P} 
(y + 1)^{\dim F} (-y)^{\mathrm{codim} F} \,|F \cap M|
\\
&= 
\nonumber
\sum_{F\in\mathscr{F}_P} (y + 1)^{\dim F} \left|F^\circ\cap M\right|.
\end{align}
\end{theorem}

\subsection{}
In \cite{paper:Materov_Bott}, the formula \eqref{eq:Bott_formula} is
called the \textit{Bott formula for toric varieties}, since it
generalizes a result due to Bott, who treated
$X = \mathbb{P}^n$, $\widehat{\Omega}_X^p = \Omega_{\mathbb{P}^n}^p$
and $\mathscr{O}_X(D) = \mathscr{O}_{\mathbb{P}^n}(a)$.  We see that
$\chi_y(X, \mathscr{O}_X(q D))$ from (\ref{eq:Bott_formula}) coincides
with $G_\varphi(q, y)$ in (\ref{eq:Gdef2}) and (\ref{eq:reciprG}) when
$\varphi\equiv 1$.  In fact, the restriction $\varphi\equiv
1$ is not necessary.  One can consider $\varphi = e^z$ as in
Section~\ref{s:proof} by working instead with the equivariant character
$\sum_p\sum_i (-1)^i\,\mathrm{Tr}(e^z, H^i(X, \widehat{\Omega}_X^p
\otimes\mathscr{O}_X(qD)))\, y^p$ of the torus $\mathbb{T}\subset X$\footnote{We
thank an anonymous referee for pointing this out to us.}.

\subsection{} Here we briefly explain work of Maxim--Sch\"urmann
\cite{paper:M-S} that studies characteristic classes of singular toric
varieties.
First, we recall the motivic Chern and Hirzebruch classes of singular
complex algebraic varieties as constructed by Brasselet--Sch\"urmann--Yokura \cite{paper:B-S-Y}.  

Let $K_0(var/X)$ be the relative Grothendieck group of complex
algebraic varieties over $X$, as introduced by Looijenga and Bittner in
relation to motivic integration, and let $G_0(X)$ be the Grothendieck group
of coherent sheaves of $\mathscr{O}_X$-modules. Then the
\textit{motivic Chern class transformation}
\[mC_y(X) : K_0(var/X) \rightarrow G_0(X)\otimes\mathbb{Z}[y]\] 
generalizes the total $\lambda$-class $\lambda^y(T^*X)$ of the
cotangent bundle to the setting of singular spaces. The \textit{un-normalized
Hirzebruch class transformation} is defined by the composition
\[T_{y*} := td_*\circ mC_y :  K_0(var/X) \rightarrow H_*(X)\otimes \mathbb{Q}[y]\]
as a class version of a $\chi_y$-genus of $X$.  Here the cohomology $H_*(X)$ denotes
either the Chow groups $A_*(X)$, or the even degree Borel--Moore homology
groups $H_{2*}^{BM}(X, \mathbb{Z})$, and
\[td_* : G_0(-)\rightarrow H_*(-)\otimes\mathbb{Q}\] is the Todd transformation.
The \textit{normalized Hirzebruch class transformation} is defined via the normalization functor $\widehat{T}_{y*} := \Psi_{(1+y)}\circ T_{y*}$, where
\[\Psi_{(1+y)}:H_*(X)\otimes\mathbb{Q}[y] \rightarrow
H_*(X)\otimes\mathbb{Q}[y,(1+y)^{-1}]\] is given in degree $k$ by
multiplication by $(1+y)^{-k}$. In fact, $\widehat{T}_{y*}$ actually takes
values in $H_*(X)\otimes\mathbb{Q}[y]$ (see \cite[Theorem
3.1]{paper:B-S-Y}); this implies, for instance, that one can set the
parameter $y$ equal to $-1$, and can thus generalize $\widehat{T}_{-1*}$ to the
total rational Chern class.

Now the \textit{motivic un-normalized} and \emph{normalized homology
Hirzebruch classes} are defined respectively as
\[
T_{y*}(X) := T_{y*}([\mathrm{id}_X]), \quad \widehat{T}_{y*}(X) := \widehat{T}_{y*}([\mathrm{id}_X]);
\]
these generalize the Hirzebruch classes of $X$ that appear in the
Hirzebruch--Riemann--Roch theorem when $X$ is smooth.  Namely, assume
$X$ is smooth of dimension $n$, and let $\{x_{j} \}$ be the Chern
roots of the tangent bundle $T_{X}$.  Then the two formal power series
\[
Q_y(x) := \frac{x(1 + y e^{-x})}{1 - e^{-x}}, \quad 
\widehat{Q}_y(x) := \frac{x(1 + y e^{-x(1+y)})}{1 - e^{-x(1+y)}} = 1 +  \frac{1-y}{2} x +\cdots
\]
define two classes 
\[
T_y^*(T_X) = \prod\limits_{j = 1}^n Q(x_j), \quad 
\widehat{T}_y^*(T_X) = \prod\limits_{j = 1}^n \widehat{Q}(x_j)\in
H^*(X)\otimes\mathbb{Q}[y],
\]
and
\[
T_{y*}(X) = T_y^*(T_X)\cap [X], \quad
\widehat{T}_{y*}(X) = \widehat{T}_y^*(T_X)\cap [X].
\]
We can now state Maxim--Sch\"urmann's result:
\begin{theorem}[Maxim--Sch\"urmann~\cite{paper:M-S}]\label{thm:ms} % [Theorem 5.4]
Let $X = X_\Sigma$ be a simplicial toric variety of dimension $n$ with
the normal fan $\Sigma = \Sigma_P$ to the polytope $P$. Suppose that
the generators of the rational cohomology (or Chow) ring of $X$ are
the classes $[D_F]$ defined by the $\mathbb{Q}$-Cartier divisors
corresponding to the faces of codimention $1$ of $P$.  Then the
normalized Hirzebruch class of $X$ is given by
\begin{eqnarray}
\label{Todd:normalized}
\widehat{T}_{y*}(X) = 
\left(
\sum_{g\in \Gamma_\Sigma}
\prod_{F\in \mathscr{F}(n-1)}
\frac{[D_F] (1 + y\, a_F(g)\, e^{-[D_F](y+1)})}
{1 - a_F(g)\, e^{-[D_F](y+1)}}
\right)\cap [X] \, .
\end{eqnarray}
\end{theorem}

\subsection{} Now we connect Theorem \ref{thm:ms} to our work.  The
main observation is that the Todd differential operator in
Theorem~\ref{thm:intro2}
\[
\mathrm{Td}_y(P, \partial/\partial h) = 
\left(
\sum_{g\in \Gamma_\Sigma}
\prod_{F\in \mathscr{F}(n-1)}
\frac{{\partial/\partial h_F} (1 + y\, a_F(g) e^{-{\partial/\partial h_F}(1+y)})}
{1 - a_F(g) e^{-{\partial/\partial h_F}(1+y)}}
\right)
\]
has the same structure as the normalized Hirzebruch class in
(\ref{Todd:normalized}). This correspondence for $y = 0$ was first
established by M.~Brion and M.~Vergne in \cite{paper:BrionVergneTodd}
(see also \cite[Theorem~13.5.6]{book:ToricVarieties}). The generic
correspondence $[D_F]\rightarrow \partial/\partial h_F$ and the
relation of $G_\varphi(q, y)$ with the polynomial $\chi_y(X,
\mathscr{O}_X(D))$ can be proved by the same technique as in
\cite[Theorem~4.5]{paper:BrionVergneTodd}; this will be published
elsewhere.

\section{Examples}

\subsection{} We conclude by giving some examples of our results.  We
begin with Theorem~\ref{thm:genrecip}.

% \subsection{}
Let $P$ be the square pyramid
% cone on a square, that is, the $3$-polytope 
with vertices $(0,0,0)$, $(1,1,1)$, $(1,-1,1)$, $(-1,1,1)$, $(-1,-1,1)$
shown in Figure~\ref{fig:coneonsquare}.    Let $q>0$ be
an integer.  We consider the generating function $G_{\varphi} (q,y)$
for different functions~$\varphi $.  
\begin{figure}[htb]
\begin{center}
\includegraphics[scale=0.3]{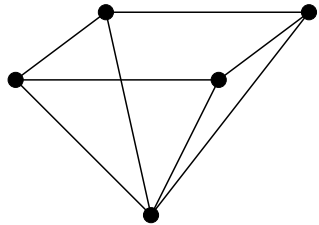}
\end{center}
\caption{The square pyramid~$P$.\label{fig:coneonsquare}}
\end{figure}

The polytope $P$ fails to be simple only at
the bottom vertex $v = (0,0,0)$.  The dual face $\PPP_{P} (v)$ is a
square, and $g_{\text{square}} (x) = 1+x$ (in general, the
$g$-polynomial of an $m$-gon is $1+ (m-3)x$).  Thus $\tilde{g}_{v}
(-\frac 1 y) = 1-\frac 1 y$, and the only effect of the non-simplicity
of $P$ is that, when we form the generating function
$G_{\varphi} (q,y)$, the contribution of the vertices
to \eqref{eq:Gdef3} is
\[
\varphi (q,q,q) + \varphi (q,-q,q) + \varphi (-q,q,q) + \varphi
(-q,-q,q) + \varphi (0,0,0) (1-\tfrac 1 y).
\]

Suppose first $\varphi = 1$.  Then 
\begin{multline*}
   G_{1} (q,y)=\Bigl(\frac{4 q^3}{3}-4 q^2+\frac{11 q}{3}-1\Bigr)
    y^3\\
+\Bigl(4 q^3-4 q^2-q+2\Bigr) y^2+\Bigl(4 q^3+4 q^2-q-2\Bigr)
    y+\frac{4 q^3}{3}+4 q^2+\frac{11 q}{3}+1 \, .
\end{multline*}
One can see the Ehrhart polynomial for $P$ in the constant term, and
that for $P\interior$ in the leading term.   It is visible that $G_{1}$
satisfies $G_{1} (q,y) = (-y)^{3} \, G_{1} (-q,\tfrac 1 y)$, and this
relation applied to the leading and constant terms is nothing other
than Ehrhart reciprocity.

Denote by $\mathrm{Vol}(P)$ the volume of polytope $P$ 
of dimension $n$ normalized so that the volume of the simplex spanned 
by the origin and  basis vectors is equal to $1$. 
Expand the polynomial $G_1(q,y)$:
\[
G_1(q,y) = \sum_{p=1}^n L_p(q) y^p.
\]
Then it is easy to see that $L_p(q)$ is the (generalized Ehrhart)
polynomial in $q$ of degree $n$ whose leading term is $\binom{n}{p}
\mathrm{Vol}(P) q^n$.  Indeed, consider the expansion
from~(\ref{eq:reciprG})
\[
G_{1} (q,y) = 
\sum_{F\in \mathscr{F} }
(y+1)^{\dim F} \EEE^\circ_{F} (q)\, \tilde{g}_{F} (-y),
%D_{1 ,F}^\circ (q)
\]
where $\EEE^\circ_{F} (q) = \card{M\cap q F^\circ} = \mathrm{Vol}(P) q^n + a_1 q^{n-1} + \cdots$,
and notice that the coefficient of the leading term of $g$-polynomial is $1$ 
according to~\ref{subs:g_pol}. 
In the example above for the square pyramid, 
$\mathrm{Vol}(P)  = 4/3$ and the highest order terms of $L_0(q)$ and $L_3(q)$
are $\frac{4}{3} q^3$, and of $L_1(q)$ and $L_2(q)$ are $4q^3$.

Next we take a linear polynomial $\varphi = ax_{1}+bx_{2}+cx_{3}$.  Note that the
symmetry of $P$ implies that we expect that the final answer should be
independent of $a$ and $b$.  Indeed, after summing over faces of $P$
we find
\begin{multline*}
G_{\varphi } (q,y) =  y^4 \Bigl(c
    q^4-\frac{10 c q^3}{3}+\frac{7 c q^2}{2}-\frac{7 c
    q}{6}\Bigr)+y^3 \Bigl(4 c q^4-\frac{20 c q^3}{3}+4
    c q^2-\frac{c q}{3}\Bigr)\\
+y^2
    \Bigl(6 c q^4+c q^2\Bigr)+y \Bigl(4 c
    q^4+\frac{20 c q^3}{3}+4 c q^2+\frac{c
    q}{3}\Bigr)+  c q^4+\frac{10 c q^3}{3}+\frac{7 c q^2}{2}+\frac{7 c q}{6} \, .
\end{multline*}
This has degree $4$ in $y$, as expected.   One can also see the
expected reciprocity law $G_{\varphi} (q,y) = (-y)^{4} \, G_{\varphi}(-q,\tfrac 1 y)$.

For the amusement of the reader, we finish with a larger example:
$\varphi = ax_{1}^{2}+bx_{2}^{2}+cx_{3}^{2}$.  The resulting $
G_{\varphi } (q,y) $ equals
\begin{multline*}
\phantom{+}y^5 \Bigl(\frac{4 a q^5}{15}-\frac{4 a q^4}{3}+\frac{7
    a q^3}{3}-\frac{5 a q^2}{3}+\frac{2 a
    q}{5}
    +\frac{4 b q^5}{15}-\frac{4 b q^4}{3}\\
    +\frac{7
    b q^3}{3}-\frac{5 b q^2}{3}+\frac{2 b
    q}{5}
    +\frac{4 c q^5}{5}-3 c q^4+\frac{11 c
    q^3}{3}-\frac{3 c q^2}{2}+\frac{c q}{30}\Bigr)
\end{multline*}
\vspace{-0.5cm}
\begin{multline*}
+y^4
    \Bigl(\frac{4 a q^5}{3}-4 a q^4+5 a q^3-3
    a q^2+\frac{2 a q}{3}+\frac{4 b q^5}{3}-4
    b q^4\\
    +5 b q^3-3 b q^2+\frac{2 b q}{3}+4
    c q^5-9 c q^4+9 c q^3-\frac{5 c
    q^2}{2}-\frac{c q}{2}\Bigr)\\
\end{multline*}
\vspace{-1cm}
\begin{multline*}
+y^3 \Bigl(\frac{8 a
    q^5}{3}-\frac{8 a q^4}{3}+\frac{10 a q^3}{3}-\frac{4
    a q^2}{3}+\frac{8 b q^5}{3}-\frac{8 b
    q^4}{3}\\
    +\frac{10 b q^3}{3}-\frac{4 b q^2}{3}+8 c
    q^5-6 c q^4+\frac{26 c q^3}{3}-c q^2-\frac{5
    c q}{3}\Bigr)\\
\end{multline*}
\vspace{-1cm}
\begin{multline*}
+y^2 \Bigl(\frac{8 a q^5}{3}+\frac{8
    a q^4}{3}+\frac{10 a q^3}{3}+\frac{4 a
    q^2}{3}+\frac{8 b q^5}{3}+\frac{8 b q^4}{3}\\
+\frac{10
    b q^3}{3}+\frac{4 b q^2}{3}+8 c q^5+6 c
    q^4+\frac{26 c q^3}{3}+c q^2-\frac{5 c
    q}{3}\Bigr)\\
\end{multline*}
\vspace{-1cm}
\begin{multline*}
+y \Bigl(\frac{4 a q^5}{3}+4 a q^4+5 a
    q^3+3 a q^2+\frac{2 a q}{3}+\frac{4 b q^5}{3}+4
    b q^4\\
+5 b q^3+3 b q^2+\frac{2 b q}{3}+4
    c q^5+9 c q^4+9 c q^3+\frac{5 c
    q^2}{2}-\frac{c q}{2}\Bigr)\\
\end{multline*}
\vspace{-1cm}
\begin{multline*}
+\frac{4 a q^5}{15}+\frac{4
    a q^4}{3}+\frac{7 a q^3}{3}+\frac{5 a
    q^2}{3}+\frac{2 a q}{5}+\frac{4 b q^5}{15}+\frac{4
    b q^4}{3}\\
+\frac{7 b q^3}{3}+\frac{5 b
    q^2}{3}+\frac{2 b q}{5}+\frac{4 c q^5}{5}+3 c
    q^4+\frac{11 c q^3}{3}+\frac{3 c q^2}{2}+\frac{c
    q}{30} \, .
\end{multline*}

\subsection{} Let $P = C_n^{\Delta}$ be the \textit{cross-polytope}
(or \textit{co-cube}).  By definition $P$ is the convex hull of the
standard basis vectors $e_1, \ldots, e_n$ and their negatives $-e_1,
\ldots, -e_n$ in $M_{\mathbb{R}}\simeq \mathbb{R}^n$.  For example,
when $n=3$, the polytope $C_{3}^{\Delta}$ is the octahedron.  It is
known that the polar dual polytope of $P$ is the unit cube $C_n$,
whose associated toric variety is isomorphic to the product
$\mathbb{P}^1\times \mathbb{P}^1\times\cdots\times\mathbb{P}^1$. The
$g$-polynomial of the cube was computed by I.~Gessel \cite[\S
2.6]{paper:Stanley_IC}:
\begin{equation}\label{eq:gessel}
g(C_n, x) = \sum_{k = 0}^m\frac{1}{n - k + 1}\binom{n}{k}\binom{2n -
2k}{n}(x -1)^k, \quad m=\lfloor n/2 \rfloor.
\end{equation}
To find the function $G_\varphi(q, y)$ of $C_n^{\Delta}$ defined in
(\ref{eq:reciprG}), we need the explicit form of the polynomial
\begin{eqnarray*}
K_F(y) := 
(1+y)^{\dim F}\; \widetilde{g}_F\left(-y\right),
\end{eqnarray*}
where $F$ is any face of $C_n^{\Delta}$.  Using \eqref{eq:gessel}, we have 
\[
K_{F} (y)= (1 + y)^{\dim F}
\sum_{k = 0}^{m_F}
\frac{1}{\dim F - k + 1}
\binom{\dim F}{k}\binom{2\dim F - 2k}{\dim F}(-y-1)^k.
\]
Putting all this together, we obtain
\[
%\frac{G_\varphi(q, y)}{(1 + y)^{\deg \varphi}} = 
G_\varphi(q, y) = 
\sum_{F\in\mathscr{F}}
\sum_{k =0}^{m_F} \frac{(-1)^k}{\dim F - k - 1}
\binom{\dim F}{k}\binom{2\dim F - 2k}{\dim F}
(y+1)^{\dim F + k - \deg \varphi} D_{\varphi, F}^\circ(q).
\]

\subsection{}
Finally we consider an example of Theorem \ref{thm:intro2}.  Let $P$ be the
triangle with vertices at $(0,0)$, $(2,0)$, and $(0,1)$.   The polygon
$P$ together with its normal fan $\Sigma$ are shown in Figure \ref{fig:polygon}.  
\begin{figure}[htb!]
 \begin{minipage}[h]{0.3\linewidth}
\center{ \includegraphics[scale=0.3]{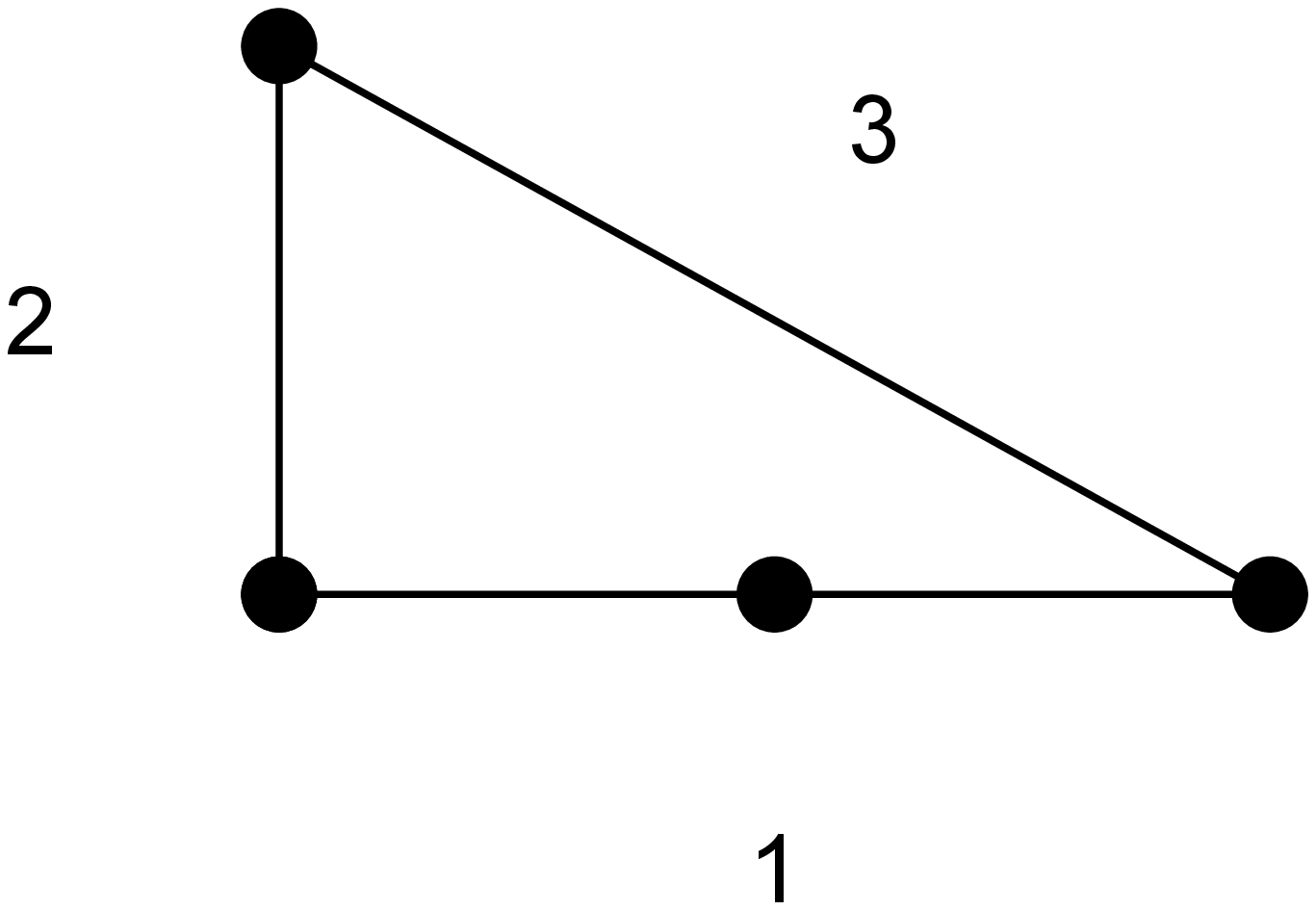} %\\ {\small }
}
 \end{minipage}
 \hfill
  \begin{minipage}[h]{0.3\linewidth}
\center{  \includegraphics[scale=0.19]{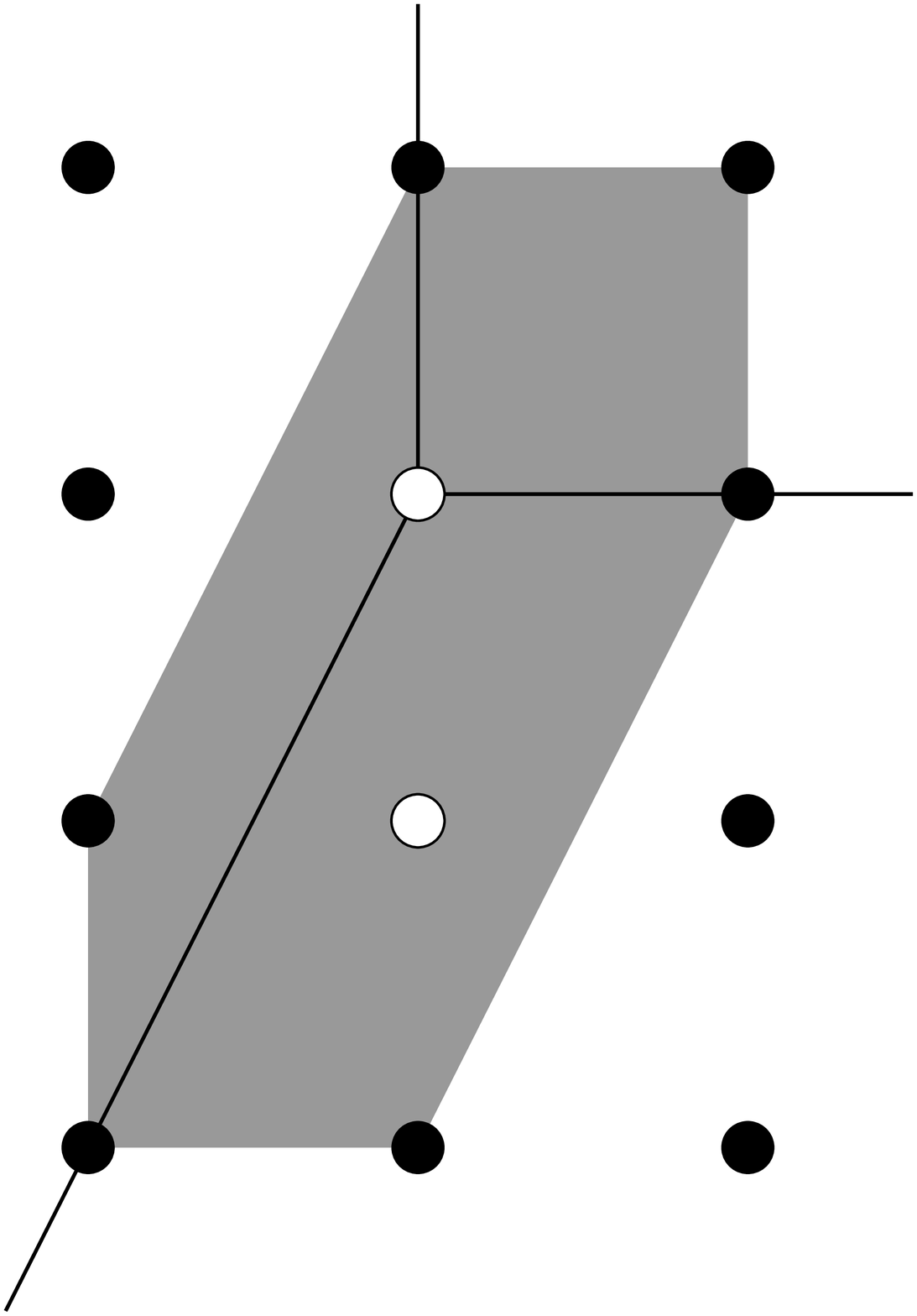} %\\ {\small }
}
 \end{minipage}
 \caption{The triangle $P$ and its normal fan $\Sigma$.}
 \label{fig:polygon}
\end{figure}
%

%
%\begin{figure}[htb]
%\psfrag{1}{$F_{1}$}
%\psfrag{2}{$F_{2}$}
%\psfrag{3}{$F_{3}$}
%\begin{center}
%\includegraphics[scale=0.3]{polygon}
%\end{center}
%%\caption{\label{fig:polygon}The triangle $P$ and its normal fan.}
%\end{figure}
%
In the normal fan the shaded regions represent the sets $Q (\sigma)$.
One can see that the set $\Gamma_{\Sigma}$ contains two lattice
points $g_{0} = (0,0)$ and $g_{1}=(0,-1)$, shown in white.   It is clear that the
all the functions $\{a_{F} \mid F \in \FFF \}$ are identically $1$ on
$g_{0}$, and that $a_{F} (g_{1}) \not = 1$ if and only if $F$ is one
of $F_{2}$ or $F_{3}$, and that for either of these $a_{F}
(g_{1}) = -1$.  Thus our Todd-$y$ operator has the form
\begin{multline}\label{eq:opexample}
\Td_{y} (P, \dee{h}) = \Td_{y} (1,\dee{h_{1}})\Td_{y}
(1,\dee{h_{2}})\Td_{y} (1,\dee{h_{3}}) \\
+ \Td_{y} (1,\dee{h_{1}})\Td_{y}
(-1,\dee{h_{2}})\Td_{y} (-1,\dee{h_{3}}) \, .
\end{multline}

First consider putting $\varphi =1$.  The function $E_{\varphi}
(\tilde{P}_q (h))$ is then just the volume of the deformed dilate
$\tilde{P}_q (h)$, which is
\begin{equation}\label{eq:vol}
\Vol \tilde{P}_q (h) = \frac{(2 {h_1}+{h_2}+{h_3}+2 q (y+1) )^2}{4} \, .
\end{equation}
Applying \eqref{eq:opexample} to \eqref{eq:vol} and putting $h_{1}=h_{2}=h_{3}=0$, we
obtain
\[
\left(q^2-2 q+1\right) y^2+\left(2 q^2-1\right) y+q^2+2 q+1 \, .
\]
It is easy to check directly that this agrees with $G_{1} (q,y)$.

Now suppose $\varphi$ is a generic homogeneous linear function
$\varphi (x_{1},x_{2}) = ax_{1}+bx_{2}$.  Then our integral becomes 
\begin{multline}\label{eq:vol2}
E_{\varphi }
(\tilde{P}_q (h)) 
= \frac{1}{24} (2 {h_1}+{h_2}+{h_3}+2 q (y+1))^2 \\
\cdot (2 {a} (2 {h_1}-2 {h_2}+{h_3}+2 q
    (y+1))+{b} (-4 {h_1}+{h_2}+{h_3}+2 q (y+1))).
\end{multline}
Applying \eqref{eq:opexample} to \eqref{eq:vol2} and setting $h_{1}=h_{2}=h_{3}=0$ yields
\begin{multline*}
(y+1)\Bigl( y^2
    \Bigl(\frac{2 {a} q^3}{3}-\frac{3 {a} q^2}{2}+\frac{5 {a} q}{6}+\frac{{b} q^3}{3}-\frac{{b}
    q^2}{2}+\frac{{b} q}{6}\Bigr) \\
+y \Bigl(\frac{4 {a} q^3}{3}-\frac{{a} q}{3}+\frac{2 {b} q^3}{3}-\frac{2 {b} q}{3}\Bigr)+\frac{2 {a} q^3}{3}+\frac{3 {a} q^2}{2}+\frac{5 {a} q}{6}+\frac{{b}
    q^3}{3}+\frac{{b} q^2}{2}+\frac{{b} q}{6}\Bigr) .
\end{multline*}

\bibliographystyle{amsplain_initials} 
\bibliography{chiyemorb}    
\end{document}